\documentclass[11pt]{article}
\usepackage[margin=1in]{geometry}
\usepackage{graphicx}
\usepackage{amsmath}
\usepackage{amsthm}
\usepackage{fancyhdr}
\newtheorem{theorem}
{Theorem}[section]
\newtheorem{corollary}
{Corollary}[theorem]
\newtheorem{lemma}[theorem]
{Lemma}
\begin{document}
\title{Numeric Palindromes in Primitive and Non-primitive\\ Pythagorean Triples\footnote{Submitted to: A peer review Mathematics Journal, February 28,2015}}
\author{John Rafael M. Antalan and Richard P. Tagle \\ \\ \\Department of Mathematics and Physics\\ College of Arts and Sciences\\ Central Luzon State University\\ Science City of Munoz, Nueva Ecija, (3120)\\ Philippines}
\maketitle
\abstract {In this article we consider numeric palindromes as components of a Pythagorean triple. We first show that there are infinitely many non-primitive Pythagorean triples that contain (i) a single numeric palindrome component, (ii) two numeric palindrome components and (iii) three numeric palindrome components. We then focus on numeric palindromes in primitive Pythagorean triples. We show that there are infinitely many primitive Pythagorean triples composed of a single and two numeric palindrome components. Open problem and preliminary results related to the open problem are also given.}
\section {Introduction and Statement of the Problem}
\ \ \ \ \ This paper is inspired by the works of Gopalan and his colleagues about Pythagorean triples found in [1] and its reference page. In [1], they determined those Pythagorean triples with a leg represented by a Kepricker number and gave some interesting results. In this paper however, we deal with numeric palindrome in Pythagorean triples.
\bigskip

A search on the web reveals that only few mathematicians and mathematics enthusiasts studied palindromes and Pythagorean triples. For instance in [2], the author studied primitive Pythagorean triples whose perimeter yields palindomic number. In [3], palindromic Pythagorean triples were studied. It can be seen in [3] that one can generate infinitely many non-primitive Pythagorean triples with three numeric palindrome components starting from the triple $(3,4,5)$.   Lastly a short discussion about Pythagorean triples was given in textbook [4] with the result similar in [3].
\bigskip          
 
In this paper however, we consider the infinitude of palindromic Pythagorean triples both primitive and non primitive having a single, double, or triple palindrome component. Our results are as follows:

\begin{enumerate}
\item There are infinitely many non-primitive Pythagorean triples with one numeric palindrome component.
\item There are infinitely many non-primitive Pythagorean triples with two numeric palindrome components.
\item There are infinitely many non-primitive Pythagorean triples with three numeric palindrome components .
\item There are infinitely many primitive Pythagorean triples with one numeric palindrome component.
\bigskip

Notice the similarity of result 3. and the result in [3] and [4]. The difference in this manuscript is that we give a different proof for it. Included in this paper is the result proved by other mathematicians (formerly our conjecture):
\bigskip
\item There are infinitely many primitive Pythagorean triples with two numericpalindrome components.

\end{enumerate}

Lastly we state an open problem related to the topic and show some preliminary results.      
\section{Preliminaries}
\ \ \ \ \ The following preliminary discussion on Pythagorean triples were taken from [5]. 
\bigskip 

A Pythagorean triple is a set of three integers $x,y,z$ such that 
\begin{center}
$x^2+y^2=z^2$.
\end{center} 
The triple is said to be primitive if $gcd(x,y,z)=1.$ Also each pair of integers $x,y,z$ are pairwise relatively prime.
\bigskip

 All of the solutions of the Pythagorean equation $x^2+y^2=z^2$ satisfying the conditions $gcd(x,y,z)=1$, $2|x$, $x,y,z>0$ are given by:

\begin{equation}
x=2st, y=s^2-t^2, z=s^2+t^2
\end{equation}  
for relatively prime integers $s>t>0$ and $s \not \equiv t(mod\ 2)$. 
\bigskip

Lastly, from a primitive Pythagorean triple $x,y,z$  a non-primitive Pythagorean triples can be generated by multiplying some positive integer constant $c$ as a result $cx,cy,cz$ forms a non-primitive Pythagorean triple.  
\bigskip

The following are  some useful notations that utilized later in the main result.
\bigskip

For a triple $(x,y,z)$ and a constant $a$, we define their product $a(x,y,z)=(ax,ay,az)$.
\bigskip

The expression appearing in the triple of the form $a--_n$ means $n$ copies of $a$. For example, if we have $11--_0$ this expression means $0$ copies of $11$ which is also $11$. Some other examples are $11--_3=11111111$ and $60--_2=606060$.
\bigskip

Lastly, if $a\in Z_{10}$, we define the expression $a_k$ as $\underbrace{aa...aa}_{k-times}$. For example if we have $13_2$  we are reffering to the number $133$. Other example is $120_23_31=12003331$.
\bigskip

With these preliminaries we are now ready to show our results.     
\section{Results}
\subsection{Numeric Palindromes in Non-primitive Pythagorean Triples}
\begin{lemma}
Any palindrome with an even number of digits is divisible by 11.
\end{lemma}
\begin{proof}
We know that a number say $n$ is divisible by 11 if and only if the alternate sum and difference of its digits is divisible by 11. For a palindrome with even number of digits say $a_0a_1a_2...a_{2n-1}a_{2n}$, we have $a_0=a_{2n}, a_1=a_{2n-1},... ,a_n=a_{n+1}$ thus, $a_0-a_1+a_2-a_3+...+a_{2n-1}-a_{2n}=0$\ which is divisible by 11.
\end{proof}
\begin{theorem}
There are infinitely many non-primitive Pythagorean triple with one numeric palindrome component.
\end{theorem}
\begin{proof}
1. To prove this theorem we note that with $s=6$ and $t=5$ in equation (1) we see that 60, 11 and 61 is a primitive Pythagorean triple. Define the number theoretic function $f$ as follows:
\begin{center}

 \[f(n)=\begin{cases}
\ 1 \ if \ n=0\\ \\
\ f(n)=(f(n-1)\cdot 10^2)+1\ if\ n\in Z^+
\end{cases}
\] 
\end{center}
and consider the product $f(n)(11,60,61)$. This product will  always generate a Pythagorean triple of the form $(11--_{n},60--,_{n},61--_{n})$, a non-primitive Pythagorean triple that contains exactly one numeric palindrome component. Since $n$ runs through the set of positive integers we conclude that there are infinitely many such triples.       
\end{proof}
\begin{proof}
2. Starting from the primitive Pythagorean triple (PPT) $(3,4,5)$,consider $3_n(3,4,5)$. This generates the non-primitive Pythagorean triple (NPPT) with one numeric palindrome component for any positive integer $n$: $(9_n,13_{n-1}2,16_{n-1}5)$.  
\end{proof}
\begin{theorem}
There are infinitely many non-primitive Pythagorean triples with two numeric palindrome components.
\end{theorem}
\begin{proof}
1. Starting from the primitive Pythagorean triple $(3,4,5)$, consider the product:
\begin{center} 
$(n^2+2n+1)(3,4,5)$ where $n=10^k, k\in Z^+$
\end{center}
 This will generate a triple of the form:
\begin{center}
\[\begin{cases}
\ (363,484,605) \ if \ k=1\\ \\
\ (30_{k-1}60_{k-1}3,40_{k-1}80_{k-1}4,50_{k-2}10_k5)\ if\ k>1.
\end{cases}
\] 
\end{center}
\bigskip

A non-primitive Pythagorean triple that contains exactly two numeric palindrome component. Since $k$ runs through the set of positive integers we conclude that there are infinitely many such triples. 

\end{proof}
\begin{proof}
 2. Starting from the primitive Pythagorean triple (PPT) $(3,4,5)$,consider $2_n(3,4,5)$. This generates the non-primitive Pythagorean triple (NPPT) with two numeric palindrome component for any positive integer $n$: $(6_n,8_n,1_n0)$.  
\end{proof}
\begin{theorem}
There are infinitely many non-primitive Pythagorean triples with three numeric palindrome components .
\end{theorem}
\begin{proof}
Starting form the primitive Pythagorean triple (3,4,5), we can form an infinite number of non-primitive Pythagorean triples with all components are palindromes by multiplying appropriate constants. In particular, we can multiply $1_k, k\in Z^+$ to the original primitive triple yeilding the non-primitive triple of the form $(3_k,4_k,5_k)$. 
\end{proof}
\subsection{Numeric Palindromes in Primitive Pythagorean Triples}
\begin{theorem}
There are infinitely many primitive Pythagorean triple with one numeric palindrome component.
\end{theorem}
\begin{proof}
In (1),let $s$ be a palindrome with digits $\in F=\{0,1,2,3,4\}$ and $t=1$ such that $s$ and $t$ satisfies the conditions in (1). It is easy to see that $x$ is a palindrome. Our claim is that the other components $y$, and $z$ are not palindromes. To see this we proceed by contradiction. Note that the difference $z-y=2$. If $y$ and $z$ where palindromes then $z=10_{n-1}1$ and $y=9_n$. Solving for $x$ in this case we see that its units digit is 0. A contradiction to the fact that $x$ is a palindrome.
\bigskip
\end{proof}

For those  primitive Pythagorean triples with two numeric palindrome components, consider table:1 (derived from [6] and [7]). Notice that there are few of them. This observation leads us to assume that there are only finite number of primitive Pythagorean triples with two numeric palindrome components. However, extending our search leads to other  primitive Pythagorean triple with two numeric palindrome components shown in table:2 (derived from [3]). In fact there are infinitely many of them as proved by Prof. Julian Aguirre in [8] and Sir T.D. Noe in [9]. Notice the difference of the two proofs. In [8] infinitude of primitive Pythagorean triple with two numeric palindrome components was established where $x$ and $y$ are palindromes while in [9] $x$ and $z$ where palindromes. 
\begin{table}[h!]
\centering
\begin{tabular}{||c c c||}
\hline 
x & y & z
\\ [0.5ex]
\hline \hline
3 & 4 & 5\\
99 & 20 & 101\\
225 & 272 & 353\\
275 & 252 & 373\\
33 & 544 & 545\\
595 & 468 & 797\\
555 & 572 & 797\\
777 & 464 & 905\\
[1ex]
 \hline
\end{tabular}
\caption{Table of primitive Pythagorean triples with two numeric palindrome components up to $s=81$ with hypotenuse less than 6000.} 
\label{table: 1}
\end{table}
\begin{table}[h!]
\centering
\begin{tabular}{||c c c||}
\hline 
x & y & z
\\ [0.5ex]
\hline \hline
313 & 48984 & 48985\\
34743 & 42824 & 55145\\
55755 & 25652  & 61373\\
52625 & 80808 & 96433\\
575575 & 2152512 & 2228137\\
5578755 & 80308  & 5579333\\
5853585 & 2532352  &  6377873\\
5679765 & 23711732 & 24382493\\
304070403 & 402080204 & 504110405\\
341484143 & 420282024 & 541524145\\
345696543 & 422282224 & 545736545\\
359575953 & 401141104 & 538710545\\
55873637855 &  27280108272  & 62177710753\\
[1ex]
 \hline
\end{tabular}
\caption{Table of other primitive Pythagorean triples with two numeric palindrome components.} 
\label{table: 2}
\end{table}
\bigskip

Lastly notice that the only primitive Pythagorean triple with all its components are palindrome is the triple $(3,4,5)$. We conjecture that this is the only such triple and leave its proof as an open problem.
\bigskip

\subsection{Open Problem}
We restate in this subsection the problem that arose in subsection 2.
\begin{itemize}
\item Prove or Disprove that $(3,4,5)$ is the only primitive Pythagorean triple with all its component are palindrome. 
\end{itemize}

\section{Primitive Pythagorean Triples with Three Numeric\\ Palindrome Components}
While it is an open problem to show the uniqueness of $(3,4,5)$ as the only primitive Pythagorean triple with three numeric palindrome component, we characterize in this section the form of others whenever they exist. 
\begin{theorem}
Primitive Pythagorean triples with three numeric palindrome components must be of the form $(E_d-O_d-O_d)$, $(O_d-E_d-O_d)$, $(O_d-O_d-E_d)$ and $(O_d-O_d-O_d)$, where $O_d$ and $E_d$ represent odd number of digits and even number of digits respectively.
\end{theorem}
\begin{proof}
Since we assumed that the triples are primitive Pythagorean triples then $gcd(x,y)=gcd(x,z)=gcd(y,z)=1$. If it happened that at least two of the components were composed of even number of digits then by Lemma 3.1 they are both divisible by 11. A contradiction to  $gcd(x,y)=gcd(x,z)=gcd(y,z)=1$ and thus a contradiction to our assumption that the triples are primitive.    
\end{proof}
The next lemma was derived from [10]. 
\begin{lemma}
In a primitive Pythagorean triple with $y$ even, and $z>x$,
\begin{enumerate}
\item Exactly one of $x$ or $y$ is divisible by $3$.
\item Leg $y$  is divisible by $4$.
\item Exactly one of $x,y,z$ is divisible by $5$. 
\end{enumerate} 
\end{lemma} 
With
Using lemma 4.2, we see that if $x,y,z$ forms a primitive Pythagorean triple, then we have for some relatively prime integers $a,b$ and $c$ the possible forms: 
\begin{table}[h!]
\centering
\begin{tabular}{||c c c||}
\hline 
x & y & z
\\ [0.5ex]
\hline \hline
15a & 4b & c\\
5a & 12b & c\\
3a & 20b & c\\
a & 60b & c\\
3a & 4b & 5c\\
a & 12b & 5c\\
[1ex]
 \hline
\end{tabular}
\caption{Table of forms of primitive Pythagorean triples.} 
\label{table: 3}
\end{table}
\bigskip

If we want to have a primitive Pythagorean triple with three numeric palindrome components we have:
\begin{theorem}
A primitive Pythagorean triple with three numeric palindrome components (whenever exists) takes the form $(5a,12b,c)$, $(15a,4b,c)$, $(3a,4b,5c)$ or $(a,12b,5c)$ for some relatively prime integers $a,b$ and $c$.
\end{theorem}
\begin{proof}
The form $(3a,20b,c)$ and $(a,60b,c)$ will never yield a primitive Pythagorean triple with three numeric palindrome components since $20b$ and $60b$ is not a palindrome.  
\end{proof}
\begin{corollary}
For a Palindromic Pythagorean Triple with three numeric palindrome components, the first and last digit of $x$ is 5 or the first and last digit of $z$ is 5.  
\end{corollary}     
\section{Conclusion}
As a conclusion of this paper we successfully showed that there are infinitely many non-primitive Pythagorean triples consisting of a single, double and triple numeric palindrome components. The case is similar for the primitive Pythagorean triple with a single and double numeric palindrome components. While a proof awaits for the uniqueness of $(3,4,5)$ as the only primitive Pythagorean triple whose all components are palindrome.
\section{Acknowledgement}
The authors are highly indebted to Gerry Myerson for his valuable comments related to the topic, Blue for his calculations that leads us the results in table 2 and lastly to Prof. Julian Aguirre of University of the Basque Country and Sir T.D. Noe of Portland Oregon for proving the infinitude of primitive pythagorean triples with double numeric palindrome component and for some helpful comments and suggestions. 
\section{Recommendation}
For future studies, aside in proving the conjecture being stated here, one may extend the idea presented here in other number bases. An extension to $n-tuples$ may also be of high interest.  
  
\end{document}